\documentclass[preprint,a4paper,11pt]{elsarticle}		

\usepackage[latin1]{inputenc}

\usepackage{amsmath}
\usepackage{amssymb,amsfonts,amsmath}
\usepackage{epsfig}
\usepackage{makeidx}
\usepackage{url}
\usepackage{indentfirst}
\usepackage{enumerate}
\usepackage{color}

\usepackage{amsthm}
\usepackage{float}

\usepackage{pstricks}
\usepackage[matrix,ps,xdvi]{xypic} 

\usepackage{a4wide}

\usepackage{multicol}
\usepackage{listings}
\usepackage{algorithm}
\usepackage{algorithmic}
\usepackage{cancel}

\newtheorem{propriete}{Proposition}

\newtheorem{corollaire}{Corollary}

\def\w{{\bf w}}
\def\T{{\mathcal T}}
\def\N{{\mathbb N}}

\newcommand{\gras}[1]{{\em  #1}}
\newcommand{\TP}{\phi}

\def\myitem{\item[$\triangleright$]}
\def\bd{\blacklozenge}

\title{Dyck Tableaux}		

\author[labri]{Jean-Christophe Aval\fnref{fn1}} 
\author[labri]{Adrien Boussicault\fnref{fn1}}
\author[liafa]{Sandrine Dasse-Hartaut\fnref{fn2}}

\fntext[fn1]{The first two authors are supported by the ANR (PSYCO project)}
\fntext[fn2]{The third  author is supported by the  ANR (IComb project)}

\address[labri]{LaBRI, Universit\'e de Bordeaux, 351 cours de la Lib\'eration, 33405 Talence, France}
\address[liafa]{LIAFA, Universit\'e Diderot - Paris 7, Case 7014, 75205 Paris Cedex 13, France}

\date{\today}

\journal{Theoretical Computer Science}

\begin{document}

\begin{abstract}
We introduce and study new combinatorial objects called Dyck tableaux which may be seen as a variant of permutation tableaux. 
These objects appear in the combinatorial interpretation of the physical model PASEP (Partially Simple Asymmetric Exclusion Process).
Dyck tableaux afford a simple recursive structure through the construction of an insertion algorithm.
With this tool, we are able to describe statistics which are relevant in the PASEP model, in a more direct way than in previous works.
Moreover, we give a new and natural link between permutations and certain labeled Dyck paths known as subdivided Laguerre histories.
\end{abstract}

\maketitle

\tableofcontents

\section*{Introduction}

The starting point of this work is the discovery of a new and direct
construction that relies bijectively the permutations of length $n$ 
to some weighted Dyck paths named {\em subdivided Laguerre histories}.
These objects correspond to the combinatorial interpretation of the 
development of the generating function for factorial numbers in 
terms of a Stieltjes continued fraction \cite{ortho}. Such a bijection
has been given by de Medicis and Viennot \cite{dMV} but their construction
is indirect in the sense that it decomposes a permutation in two involutions,
then goes through the construction and the fusion
of two Hermite histories.

Another interest of our construction is that it gives a link between
subdivided Laguerre histories and
tree-like tableaux \cite{TLT}, which are a new presentation
of permutation tableaux \cite{PT,SW} or alternative tableaux \cite{AT}.
The link lies in the insertion algorithm used on both classes
of objects and whose key ingredient is the notion of {\em ribbon}.
For this reason, the central objects of this paper are tableaux
called {\em Dyck tableaux} whose natural reading in terms of words
gives subdivided Laguerre histories. Although the original construction
is not recursive, we are able to easily describe relevant statistics
(generalized patterns, shape, (RL/LR)-(minima/maxima))
because of the recursive structure given by the insertion procedure.

When talking about relevant statistics, we have in mind the long-term
challenging motivation of this work:  build new objects 
in order to give a new, and if possible simpler, interpretation 
of the statistics introduced by Corteel and Williams \cite{CW2} 
which describe the stationary state of the physical model named PASEP. 
The {\em Partially Asymmetric Simple Exclusion Process} is a model 
in which $N$ sites on a one-dimensional lattice are either empty
or occupied by a single particle. These particles may hop to the
left or to the right with fixed probalities, which defines
a Markov chain on the $2^N$ states of the model.
The explicit description of the stationary probability of the PASEP
was obtained through the Matrix-Ansatz \cite{derrida}.
Since then, the links between specializations of this model 
and combinatorics have been the subject of an important 
research (see for example \cite{DS,CW1,CN,MJV}).
A great achievment is the description of the stationary distribution
of the most general PASEP model through statistics defined on 
combinatorial objects called staircase tableaux \cite{CW2}. Although this result
gives an explicit solution, it would be valuable to give another,
and simpler interpretation.

This paper is divided into eight sections.
The first one is devoted to the definition of Dyck tableaux,
together with their word reading.
In Section \ref{sec:insertion},
we present the insertion algorithm which gives Dyck tableaux
their nice recursive structure. 
Section \ref{sec:bij-permutations} presents the direct bijection between
Dyck tableaux and permutations.
Section \ref{sec:patterns} studies some generalized permutation patterns.
The bijection with tree-like tableaux is presented in Section \ref{sec:bij-TLT}.
The next two sections deal with relevant parameters of Dyck tableaux:
their shape in Section \ref{sec:shape}, and (RL/LR)-(minima/maxima) 
in Section \ref{sec:min-max}. To conclude, the last section points out
open questions.

\section{Dyck tableaux}
\label{sec:definitions}

We shall call \gras{staircase partition} of size $n$ the partition 
$E_n=(n,n-1, \dots, 1)$. As usual, a partition is represented by a Ferrers 
diagram and Figure \ref{fig:partition_escalier} illustrates the convention
that we choose to draw such diagrams.
\begin{figure}[H]
$$
\begin{array}{c}
\mbox{\epsfig{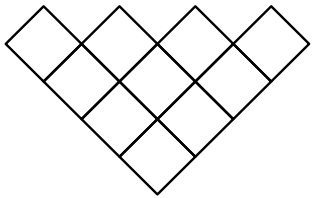}}
\end{array}
$$
\caption{The staircase partition $E_4$ \label{fig:partition_escalier}}
\end{figure}
The set of {\em Dyck paths} of size $n$ may be defined as  
$$
D_n = \{ E_n /  \mu \ | \  \mu \subset E_{n-1} \}
$$
where $/$ denotes the suppression of a partition box by box.
Figure \ref{fig:tableaux_de_dyck_3} shows the set of the $5$
Dyck paths of size $3$.
\begin{figure}[H]
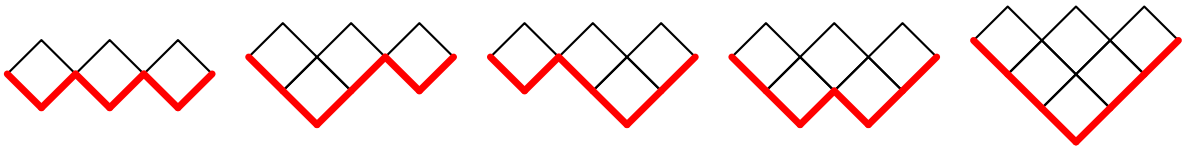

$$
\begin{array}{c}
\includegraphics[scale=0.9]{images/tableaux_de_dyck_3.1}
\end{array}
\begin{array}{c}
\includegraphics[scale=0.9]{images/tableaux_de_dyck_3.2}
\end{array}
\begin{array}{c}
\includegraphics[scale=0.9]{images/tableaux_de_dyck_3.3}
\end{array}
\begin{array}{c}
\includegraphics[scale=0.9]{images/tableaux_de_dyck_3.4}
\end{array}
\begin{array}{c}
\includegraphics[scale=0.9]{images/tableaux_de_dyck_3.5}
\end{array}
$$
\caption{The Dyck paths of size $3$\label{fig:tableaux_de_dyck_3}}
\end{figure}
It is more convenient for our purpose to use diagrams, but of course, 
the lower border of such a diagram $\pi$ is an "usual" Dyck paths. 
Moreover, when reading on the border of $\pi$
(from left to right) a $D$ letter for a step going down, and an $U$
letter for a step going up, we bijectively get {\em Dyck words}.
For example, the Dyck word associated to the second Dyck path $\pi$ 
on Figure \ref{fig:tableaux_de_dyck_3} is $D\, D\, U\, U\, D\, U$.
In a Dyck path $\pi$ of size $n$, the set of $n$ boxes placed at the top 
is called its \gras{first floor}. This notion is illustrated 
by Figure \ref{fig:base_du_tableaux_de_dyck}.
\begin{figure}[H]
$$
\begin{array}{c}
\includegraphics{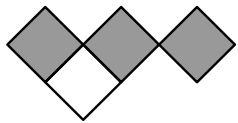}
\end{array}
$$
\caption{First floor of a Dyck path\label{fig:base_du_tableaux_de_dyck}}
\end{figure}
We label the boxes of the first floor from left to right,
and from $1$ to $n$.
The set of boxes of $\pi$ on the same vertical
as the $i$-th box of the first floor is called the
\gras{$i$-th column} of $\pi$, as shown by Figure \ref{fig:numerotation_des_colones}.
\begin{figure}[H]
$$
\begin{array}{c}
\includegraphics{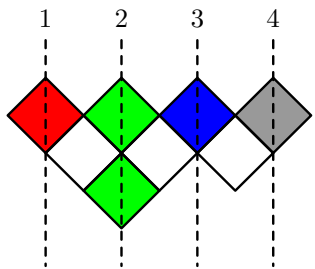}
\end{array}
$$
\caption{The labeling of columns\label{fig:numerotation_des_colones}}
\end{figure}
The number of boxes in a given column is called its \gras{height}; for
example, the column $2$ on Figure \ref{fig:numerotation_des_colones} 
is of height $2$.

Now we may define the central objects in our work.
A \gras{Dyck tableau} is a Dyck path whose columns contain each 
exactly one dot (dotted box).
The size of a Dyck tableau is its number of dots, which coincides
with the size of its underlying Dyck path.
Figure \ref{fig:tableau_de_dyck_value} is an example of Dyck tableau 
of size $6$ and Figure \ref{fig:tableaux_de_dyck_values_de_taille_3}
presents all the Dyck tableaux of size $3$.
\begin{figure}[H]
$$
\begin{array}{c}
\includegraphics{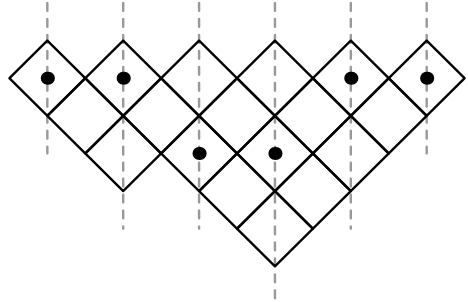}
\end{array}
$$
\caption{A Dyck tableau of size $6$\label{fig:tableau_de_dyck_value}}
\end{figure}
\begin{figure}[H]
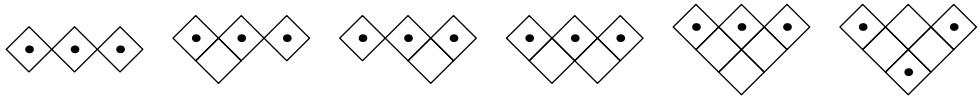

$$
\begin{array}{c}
\includegraphics[scale=0.6]{images/tableaux_de_dyck_values_de_taille_3.1}
\end{array}
\begin{array}{c}
\includegraphics[scale=0.6]{images/tableaux_de_dyck_values_de_taille_3.2}
\end{array}
\begin{array}{c}
\includegraphics[scale=0.6]{images/tableaux_de_dyck_values_de_taille_3.3}
\end{array}
\begin{array}{c}
\includegraphics[scale=0.6]{images/tableaux_de_dyck_values_de_taille_3.4}
\end{array}
\begin{array}{c}
\includegraphics[scale=0.6]{images/tableaux_de_dyck_values_de_taille_3.5}
\end{array}
\begin{array}{c}
\includegraphics[scale=0.6]{images/tableaux_de_dyck_values_de_taille_3.6}
\end{array}
$$
\caption{The Dyck tableaux of size $3$\label{fig:tableaux_de_dyck_values_de_taille_3}}
\end{figure}
The \gras{height of a dot} in a Dyck tableau is defined as the number
of empty boxes above it in the same column.
For example, on Figure \ref{fig:tableau_de_dyck_value}, the dot in 
column $2$ (resp. $3$) is at height $0$ (resp. $1$).
We need to define the \gras{basement} of size $n$ as the diagram:
$$
S_n = E_{n} / E_{n-1}.
$$
This notion is illustrated on Figure \ref{fig:socle}.
\begin{figure}[H]
$$
\begin{array}{c}
\includegraphics{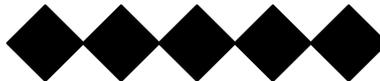}
\end{array}
$$
\caption{Basement of size $5$\label{fig:socle}}
\end{figure}
We shall now represent Dyck paths and Dyck tableaux of size $n$ 
with a basement $S_{n+1}$ whose boxes are colored in black,
as shown on Figure \ref{fig:tableau_de_dyck_value_reposant_sur_son_socle}.
\begin{figure}[H]
$$
\begin{array}{c}
\includegraphics{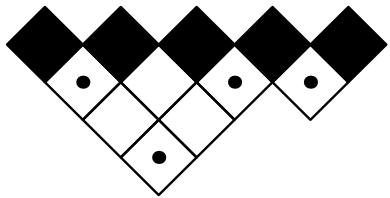}
\end{array}
$$
\caption{A Dyck tableau and its basement\label{fig:tableau_de_dyck_value_reposant_sur_son_socle}}
\end{figure}
%
It is sometimes convenient to use a word representation for Dyck tableaux.
A \gras{weighted Dyck word} is a word $w$ with letters in 
$\{\bd,U,D\}\cup \mathbb{N}$ such that:
\begin{itemize}
\myitem the word $w$ belongs to the language defined by
$$\left(\bd(U+D)\mathbb{N}(U+D)\right)^*\bd;$$
\myitem the sub-word in the letters $D$ and $U$ is a Dyck word;
\myitem for each position $i$,
$$
w(i) \in \mathbb{N} \ \ \Longrightarrow \ \ w(i) < ch(i,w)
$$
where $ch(i,w)$ is the {\em column height} defined by
$$
ch(i,w):= \left\lceil \frac{|\{ j<i | w(j)=D \}| - |\{ j<i | w(j)=U \}| }{2} \right\rceil .
$$
\end{itemize}
The integer entries of a weighted Dyck word are called its \gras{weights}. 
The same notion of weighted paths appears in \cite{dMV} under the name
{\em subdivided Laguerre histories}.
It should be clear that Dyck tableaux and weighted Dyck paths
are two representations of the same object, as illustrated
on Figure \ref{fig:bijection_tableaux_de_dyck_values_mots_de_dyck_values}:
to get a word from a Dyck tableau, one has to read the tableau (with its
basement) from left to right and to write:
\begin{itemize}
\myitem $\bd$ for a box of the basement;
\myitem $D$ for a down step of the border;
\myitem $U$ for an up step of the border;
\myitem $i$ for a dot at height $i$.
\end{itemize}
\begin{figure}[H]
$$
\begin{array}{c}
\includegraphics[scale=.8]{images/bijection_tableaux_de_dyck_values_mots_de_dyck_values.1}
\end{array}
$$
\caption{Bijection between Dyck tableaux and weighted Dyck paths\label{fig:bijection_tableaux_de_dyck_values_mots_de_dyck_values}}
\end{figure}

\section{Insertion procedure}
\label{sec:insertion}

The insertion procedure is the way to insert a dotted box inside a Dyck tableau, thus giving to the set of Dyck tableaux
a nice recursive structure.
In this section we define this insertion procedure and we present a  generation tree for Dyck tableaux.
The insertion procedure uses two main operations: column addition and ribbon addition.
Let us describe these two operations.

Let $T$ be a Dyck tableau and $w$ its reading as a weighted Dyck word.
We call \gras{column addition} in $w$ the substitution
$$
\bd \longrightarrow \bd D m U \bd
$$
where $m=ch(i,w)-1$ and $i$ is the position of the substituted $\bd$ letter in $w$.
For example, a column addition at the third $\bd$ letter in the word
$$
\bd D 0 D \bd D 1 D \textcolor{red}{\underline{{\bf \bd}}} U 0 U \bd U 0 D \bd U 0 U \bd
$$
gives the following word:
$$
\bd D 0 D \bd D 1 D \textcolor{red}{\underline{{\bf \bd D 2 U \bd }}} U 0 U \bd U 0 D \bd U 0 U \bd .
$$
Figure \ref{fig:inserer_colonne} illustrates this example on Dyck tableaux.
\begin{figure}[H]
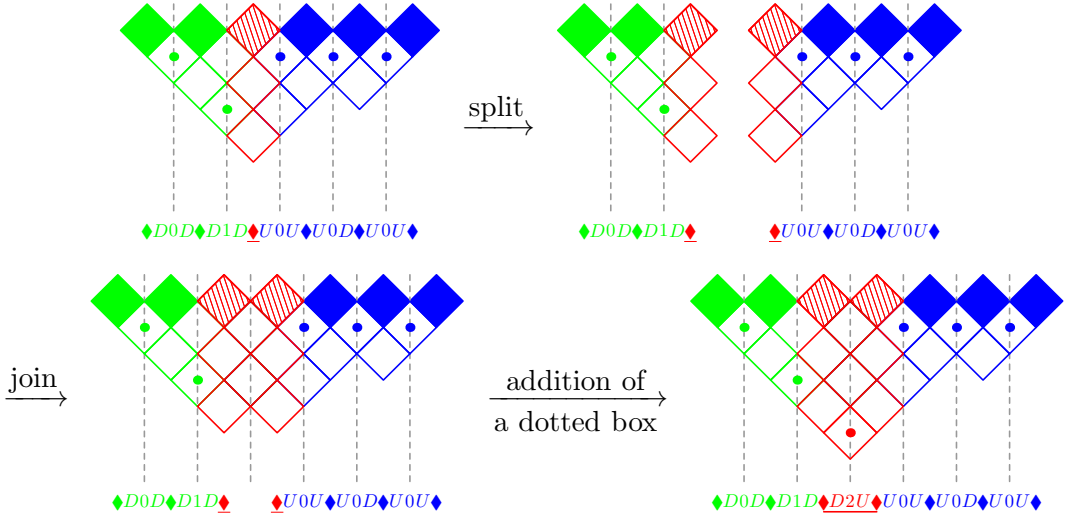

$$
\begin{array}{c}
\includegraphics[scale=.7]{images/inserer_colonne.1}
\end{array}
\xrightarrow{\mbox{split}}
\begin{array}{c}
\includegraphics[scale=.7]{images/inserer_colonne.2}
\end{array}
\begin{array}{c}
\includegraphics[scale=.7]{images/inserer_colonne.3}
\end{array}
$$
$$
\xrightarrow{\mbox{join}}
\begin{array}{c}
\includegraphics[scale=.7]{images/inserer_colonne.4}
\end{array}
\xrightarrow[\mbox{a dotted box}]{\mbox{addition of}}
\begin{array}{c}
\includegraphics[scale=.7]{images/inserer_colonne.5}
\end{array}
$$
\caption{Column addition\label{fig:inserer_colonne}}
\end{figure}

\begin{propriete}
\label{prop:insertion_colonne_mot_de_dey_value}
A column addition transforms a weighted Dyck word of size $n$ into a weighted Dyck word of size $n+1$.
\end{propriete}
\begin{proof}
Let $w=w_1 \bd w_2$ be a weighted Dyck word and $w'=w_1 \bd D m U \bd w_2$ be the result of a column insertion in $w$.
Let $i$ be the position of $m$ in $w'$.
Let $L$ be the language $\left(\bd(U+D)\mathbb{N}(U+D)\right)^*\bd$.

Let us check that  $w'$ verifies the 3 properties defining weighted Dyck words:
\begin{itemize}
\myitem If $w_1 \bd w_2 \in L$, then, by a recursion argument, we get that $w_1 \bd \in L$ and that $\bd w_2 \in L$.
Hence, $w_1 \bd D m U \bd w_2 \in L$.
\myitem The insertion of $DU$ in a Dyck word does not change its nature.
Hence, the subword in $U$ and $D$ letters of $w'$ is a Dyck word.
\myitem 
The insertion of $DU$ does not modify the  weight and the height of the existing columns.
Moreover, the new weight is such that $m<ch(i,w)$ because $m=ch(i,w)-1$.
\end{itemize}
It is clear that the size of a weighted Dyck path is increased by $1$ in a column addition.
\end{proof}

Let $w$ be a weighted Dyck word.
The \gras{ribbon addition} on the subword $U\,D$ in $w$ is the operation exchanging the two letters $D$ and $U$  of the subword in $w$.

For example, the word
$$
\bd D 0 D \bd D 1 \textcolor{red}{\underline{{\bf U}}} \bd U 0 D \bd U 0 U \bd D 0 D \bd U 0 U \bd  \textcolor{red}{\underline{{\bf D}}} 0 U \bd
$$
has a subword $U\, D$ with a $U$ letter in the $8$-th position and a $D$ letter in the $26$-th position.
If we add a ribbon between these two letters, we obtain the word : 
$$
\bd D 0 D \bd D 1 \textcolor{red}{\underline{{\bf D}}} \bd U 0 D \bd U 0 U \bd D 0 D \bd U 0 U \bd  \textcolor{red}{\underline{{\bf U}}} 0 U \bd.
$$
Figure \ref{fig:ajouter_ruban} illustrates this example on Dyck tableaux.
\begin{figure}[H]
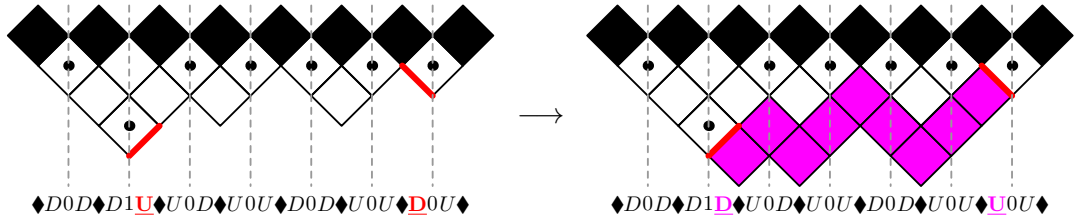

$$
\begin{array}{c}
\includegraphics[scale=.8]{images/ajouter_ruban.1}
\end{array}
\longrightarrow
\begin{array}{c}
\includegraphics[scale=.8]{images/ajouter_ruban.2}
\end{array}
$$
\caption{A ribbon addition \label{fig:ajouter_ruban}}
\end{figure}

\begin{propriete}
\label{prop:ajout_ruban_mot_de_dyck_value}
A ribbon addition transforms a weighted Dyck word into a weighted Dyck word of same size. 
\end{propriete}

\begin{proof}
Let $w$ be a weighted Dyck word.
Let $U$ and $D$ be two letters of $w$ such that $U$ is placed before $D$ in $w$.
Permuting $U$ and $D$ in $w$
\begin{itemize}
\myitem transforms the underlying Dyck word in $U$ and $D$ of $w$ in another Dyck word;
\myitem does not change the weights of $w$;
\myitem just increases by $1$ the height of the columns placed between $U$ and $D$ in $w$.
\end{itemize}
Hence, the new word verifies all the conditions to be a weighted Dyck word.
\end{proof}

Before defining the insertion procedure, we need to define the special box of a Dyck tableau.
In a Dyck tableau, an \gras{eligible box} is a dotted box with no box to its South-West.
The \gras{special box} is the right-most eligible box.
Figure \ref{fig:case_speciale} gives an example of special box.
\begin{figure}[H]
$$
\begin{array}{c}
\includegraphics{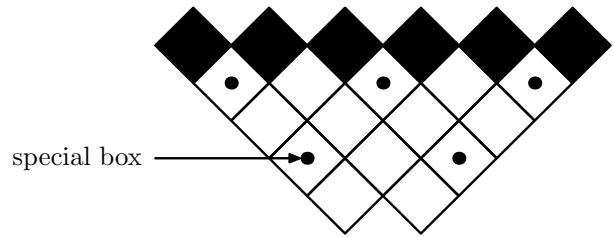}
\end{array}
$$
\caption{Special box of a Dyck tableau\label{fig:case_speciale}}
\end{figure}

In weighted Dyck words, an eligible box corresponds to a maximal weight with a $D$ letter just to its left.
We call these weights \gras{eligible}.
The special box corresponds to the right-most eligible weight, called the \gras{special weight}.
From now on, we shall use a $s$ letter to indicate a special box or a special weight.

\begin{propriete}
\label{prop:case_speciale}
A weighted Dyck word of size $k \ge 1$ has always a unique special weight.
\end{propriete}
\begin{proof}
A weighted Dyck word of size $k \ge 1$ has at least one weight.
Its first three letters are always : $\bd D 0$ and the height of the first column is $1$.
We deduce that the first weight is a eligible.
The uniqueness is obvious.
\end{proof}

We are now ready to present the insertion procedure.
The \gras{insertion procedure} is an algorithm to insert (in an invertible way)  a dotted box (and therefore a column) in a Dyck tableau.
This procedure is composed of four steps: 
\begin{algorithm}[H]
\caption{Insertion procedure}
\label{alg:algorithme_insertion}
\begin{algorithmic}[1]
\REQUIRE a weighted Dyck word of size $k\ge0$

If the size of the weighted Dyck word is 0 then ignore step \ref{etape_reperer_case_speciale} and \ref{etape_ajouter_ruban}.
\STATE \label{etape_reperer_case_speciale} Find the special weight $s$.
\STATE \label{etape_choisir_arete} Choose a $\bd$ letter.
\STATE \label{etape_inserer_colonne} Add a column at the position of the chosen $\bd$;
\STATE \label{etape_ajouter_ruban} If the chosen $\bd$ is to the left of $s$, perform a  ribbon addition to the $U$ letter following the new weight and the $D$ letter preceding $s$.
\ENSURE a final weighted Dyck word of size $k+1$.
\end{algorithmic}
\end{algorithm}

Figure \ref{fig:algorithme_insertion_tableau_de_dyck_value_ruban} gives a complete example of insertion procedure where we have to add a ribbon
at  step \ref{etape_ajouter_ruban}.
\begin{figure}[H]
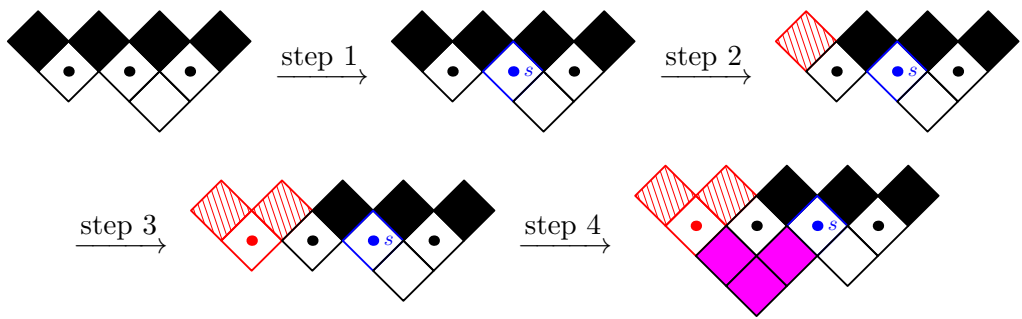

$$
\begin{array}{c}
\includegraphics[scale=.8]{images/algorithme_insertion_tableau_de_dyck_value.11}
\end{array}
\xrightarrow{\mbox{step \ref{etape_reperer_case_speciale}}}
\begin{array}{c}
\includegraphics[scale=.8]{images/algorithme_insertion_tableau_de_dyck_value.12}
\end{array}
\xrightarrow{\mbox{step \ref{etape_choisir_arete}}}
\begin{array}{c}
\includegraphics[scale=.8]{images/algorithme_insertion_tableau_de_dyck_value.13}
\end{array}
$$
$$
\xrightarrow{\mbox{step \ref{etape_inserer_colonne}}}
\begin{array}{c}
\includegraphics[scale=.8]{images/algorithme_insertion_tableau_de_dyck_value.14}
\end{array}
\xrightarrow{\mbox{step \ref{etape_ajouter_ruban}}}
\begin{array}{c}
\includegraphics[scale=.8]{images/algorithme_insertion_tableau_de_dyck_value.15}
\end{array}
$$
\caption{Insertion procedure \label{fig:algorithme_insertion_tableau_de_dyck_value_ruban}}
\end{figure}
Figure \ref{fig:algorithme_insertion_tableau_de_dyck_value} gives a complete example of insertion procedure where we do not have to add a ribbon during step \ref{etape_ajouter_ruban}.
\begin{figure}[H]
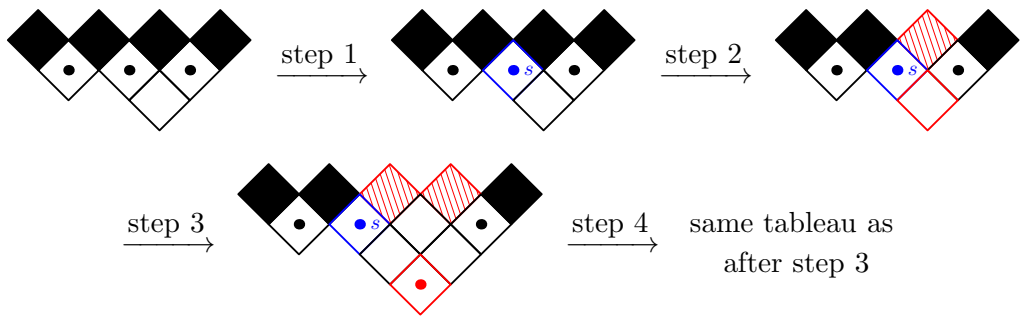

$$
\begin{array}{c}
\includegraphics[scale=.8]{images/algorithme_insertion_tableau_de_dyck_value.1}
\end{array}
\xrightarrow{\mbox{step \ref{etape_reperer_case_speciale}}}
\begin{array}{c}
\includegraphics[scale=.8]{images/algorithme_insertion_tableau_de_dyck_value.2}
\end{array}
\xrightarrow{\mbox{step \ref{etape_choisir_arete}}}
\begin{array}{c}
\includegraphics[scale=.8]{images/algorithme_insertion_tableau_de_dyck_value.3}
\end{array}
$$
$$
\xrightarrow{\mbox{step \ref{etape_inserer_colonne}}}
\begin{array}{c}
\includegraphics[scale=.8]{images/algorithme_insertion_tableau_de_dyck_value.4}
\end{array}
\xrightarrow{\mbox{step \ref{etape_ajouter_ruban}}}
\hspace{.25cm} { \mbox{same tableau as } \atop \mbox{after step 3} }
$$
\caption{Insertion procedure \label{fig:algorithme_insertion_tableau_de_dyck_value}}
\end{figure}

\begin{proof}[Insertion algorithm validity]
Let $w$ be a weighted Dyck word of size $n \ge 0$. The case $n=0$ is obvious, thus we suppose $n\ge 1$.
Step \ref{etape_reperer_case_speciale} is correct because there is always a special weight in a weighted Dyck word (Proposition \ref{prop:case_speciale}).
By definition, a weighted Dyck word has at least one $\bd$ letter, whence step \ref{etape_choisir_arete} is valid, 
and step \ref{etape_inserer_colonne} can be performed.
Thanks to Proposition \ref{prop:insertion_colonne_mot_de_dey_value}, the result of step \ref{etape_inserer_colonne} 
is  a weighted Dyck word of size $n+1$.

From the column addition, we know that the letter following the new weight is a $U$,
and from the definition of the special weight, we know that the letter preceding of $s$ is a $D$.
Now, if the new weight is on the left of $s$, we get a subword $U\,D$ on which we can add a ribbon.
We deduce that  step \ref{etape_ajouter_ruban} is well defined.

Because of Proposition \ref{prop:ajout_ruban_mot_de_dyck_value}, we conclude that the output of the algorithm 
is a weighted Dyck word of size $n+1$.
\end{proof}

The next proposition is a key ingredient.
\begin{propriete}
\label{prop:nouvelle_case_egal_case_speciale}
At the end of the insertion procedure, the new weight becomes the special weight.
\end{propriete}

\begin{proof}
Let $w$ be a Dyck word.
We may write  $w=w_1 s w_2$, with  $w_1$ and $w_2$ two words and $s$ the special weight of $w$.
Let $w'$ be the result of the insertion procedure applied to $w$.
Step \ref{etape_inserer_colonne} consists in a column addition.
When we add a column, we substitute a $\bd$ letter in $w$ by the word $\bd D m U \bd$ where $m=ch(i,w)-1$ is the new weight.
The new weight is maximal and have a $D$ letter to its left.
Thus the new weight $m$ is eligible.

We want to prove now that $m$ is the right-most eligible weight.
Two cases have to be considered:
\begin{enumerate}[1)]
\item  $m$ is to the right of $s$: as $s$ is the special weight in $w$, there is no eligible weight to its right,
whence no eligible weight to the right of $m$;
\item  $m$ is to the left of  $s$: a ribbon is added between $m$ and $s$ at step  \ref{etape_ajouter_ruban}.
Columns between $m$ and $s$ have their height increased by $1$ and their weight unchanged.
Hence, all the weights between $m$ and $s$ are not maximal.
As $s$ is the special weight in $w$, the weights placed to its right are not eligible.
We deduce that there is no eligible weight to the right of $m$.
\end{enumerate}
We conclude that $m$ is the special weight.
\end{proof}

This insertion would be of no use if it could not be inversed.
We shall now describe the \gras{inverse insertion procedure}, which consists in three steps: 
\begin{algorithm}[H]
\caption{Inverse insertion procedure}
\label{alg:algorithme_insertion_inverse}
\begin{algorithmic}[1]
\REQUIRE a weighted Dyck word of size $k \ge 1$
\STATE \label{etape_inverse_repere_case_speciale} Find the special weight $s$.
\STATE \label{etape_inverse_suprimmer_ruban} If the letter following  $s$ is a $D$, then find the leftmost maximal weight $m$ 
placed to the right of $s$.
The letter preceding  $m$ has to be a $U$. 
Delete the ribbon between $D$ and $U$ by permuting these two letters.
\STATE \label{etape_inverse_suprimer_colonne} Delete the column of $s$ : substitute the factor $\bd D s U \bd$ by the $\bd$ letter.
\ENSURE a final valued Dyck word of size $k-1$.
\end{algorithmic}
\end{algorithm}

\begin{proof}[Inverse insertion procedure validity]
Let $w$ be a weighted Dyck word of size $k \geq 1$.
We will now justify the three steps of the algorithm:
\begin{enumerate}[1)]
\item Step \ref{etape_inverse_repere_case_speciale} is correct because there is always a special weight in $w$ (Proposition \ref{prop:case_speciale}).
\item We may suppose that the the letter following  $s$ is $D$,  since the other case is trivial.
We check successively: that we can perform step \ref{etape_inverse_suprimmer_ruban}
and that it gives a weighted Dyck word of size $k$.
\begin{enumerate}[i)] 
\item To perform  step \ref{etape_inverse_suprimmer_ruban}, we need to find the first maximal weight $m$ placed to the right of $s$.
This weight exists since the rightmost weight of a weighted Dyck word is always maximal and can not be special 
(the letter just to its right is a $U$).
The letter preceding $m$ is a $U$: otherwise $m$ would be an eligible weight placed to the right of  $s$.
Thus we can apply  step \ref{etape_inverse_suprimmer_ruban}.
\item By performing step \ref{etape_inverse_suprimmer_ruban}, we permute the letters $D$ and $U$ to obtain a new word $w'$.
Let $i$ and $j$ be the position of $s$ and $m$ in $w$.
We know that $m$ is the leftmost maximal weight placed to the right of $s$.
We deduce that
$$
\forall k \in ]i,j[, \ \ w(k) \in \mathbb{N} \Rightarrow w(k) \le ch(k,w)-2.
$$
When we permute the letters $D$ and $U$, we decrease by $1$  the height of all the columns placed between $s$ and $m$.
Thus, the new word $w'$ has the following properties:
\begin{itemize}
\myitem the weights between $s$ and $m$ are strictly smaller than the height of their columns;
\myitem the subword in $U$ and $D$ of $w'$ is a Dyck word.
\end{itemize}
We deduce that the result of step \ref{etape_inverse_suprimmer_ruban} is a weighted Dyck word $w'$ of size $k$.
\end{enumerate}
\item 
By definition of the special weight,  the letter preceding $s$ in $w$ (thus in $w'$) is a $D$.
In the two cases of step \ref{etape_inverse_suprimmer_ruban}, the  letter following $s$ in $w'$  is a $U$.
We deduce that $\bd D s U \bd$ is a factor of $w'$.
We can perform step \ref{etape_inverse_suprimer_colonne} and substitute the factor $\bd D s U \bd$ by a $\bd$ letter
to get a word $w''$.
This substitution does not modify the height or the weight of the columns.
Moreover, deleting a factor $DU$ in a Dyck word gives another Dyck word.
We can conclude that $w''$ is a weighted Dyck word of size $k-1$.
\end{enumerate}
\end{proof}

\begin{propriete}
\label{prop:inverse_de_algorithme_d'insertion}
The inverse insertion procedure is the inverse of the insertion procedure.
\end{propriete}

\begin{proof}
Let $w$ be a weighted Dyck word and $p$ a $\bd$ letter of $w$. 
We denote by $w'$ the result of the insertion procedure applied to letter $p$ in $w$.
We want to check that the inverse insertion procedure appied to $w'$ gives back $w$.
Two cases have to be distinguished, according to the place of $s$, the special weight of $w$:
\begin{description}
\item{{\it Case 1:}} {\it $p$ is to the left of $s$.}\\
In this case, we may write $w=w_1 p w_2 D s w_3$, with $w_1$, $w_2$ and $w_3$ factors of $w$.
Thus $w'=w_1 \bd D m D \bd w_2 U s w_3$ where $m$ is the inserted weight.
We now apply the inverse insertion procedure to $w'$.
Thanks to Proposition \ref{prop:nouvelle_case_egal_case_speciale}, the special weight of $w'$ is $m$.
Since there is a $D$  following $m$, we have to perform step \ref{etape_inverse_suprimmer_ruban} and remove a ribbon.
To do this, we first search for the leftmost maximal weight placed to the right of $m$.
This maximal weight is $s$: indeed, by permuting $U$ and $D$ during the insertion procedure, we have  increased the height 
of all the columns between $m$ and $s$ and leave their weight unchanged.
By removing the ribbon,
we get the word $w_1 \bd D m U \bd w_2 D s w_3$.
At step \ref{etape_inverse_suprimer_colonne}, we replace the factor $\bd D m U \bd$ by a $\bd$ letter.
We finally get $w_1 \bd w_2 D s w_3$ which is precisely $w$.
\item{{\it Case 2:}} {\it $p$ is to the right of $s$.}\\
In this case, we may write $w=w_1 s w_2 p w_3$, with $w_1$, $w_2$ and $w_3$  factors of $w$.
We get  $w'=w_1 s w_2 \bd D m U \bd w_3$ where $m$ is the inserted weight.
We now apply the inverse insertion procedure to $w'$.
Using Proposition \ref{prop:nouvelle_case_egal_case_speciale}, we select the weight $m$ at step \ref{etape_inverse_repere_case_speciale}.
At  step \ref{etape_inverse_suprimmer_ruban}, there is nothing to do, because the letter following $m$ is a $U$.
At  step \ref{etape_inverse_suprimer_colonne} we replace the factor $\bd D m U \bd$ by a letter $\bd$.
We finally get $w_1 s w_2 \bd w_3$ which is precisely $w$.
\end{description}
If we apply the inverse insertion procedure to a weighted Dyck word $w$ of size $n$,
we get  a Dyck path $w'$ of size $n-1$ and an integer $0\le k\le n-1$ (the place of the special weight of $w$).
We prove in a similar way that if we apply the insertion  procedure to $w'$ at place $k$, we get back $w$.
\end{proof}

\begin{propriete}
\label{prop:generation_mot_de_dyck_value}
Every weighted Dyck word $w$ may be constructed from $\bd$ (the word of size $0$), recursively using the insertion procedure.
\end{propriete}

\begin{proof}
We shall prove the result by induction on the size of $w$.
It is obvious if the size is $0$, and we assume it holds up to size $k$.
Let $w$ be a weighted Dyck word of size $k+1$.
If we use the inverse insertion procedure on $w$, we obtain a weighted Dyck word $w'$ of size $k$.
By induction, we know that $w'$ may be obtained from the word $\bd$ using the insertion procedure.
Thanks to Proposition \ref{prop:inverse_de_algorithme_d'insertion}, $w$ can be constructed by applying the insertion algorithm to $w'$.
\end{proof}

The generation tree for Dyck tableaux of size at most $3$ is shown on Figure \ref{fig:generation_tableau_dyck}.
In this tree, an arrow of label $i$ links a tableau $T$ to a tableau $T'$ obtained by inserting a dot in the $i$-th box of the basement of $T$.
\begin{figure}[H]
$$
\begin{array}{c}
\includegraphics[scale=.6]{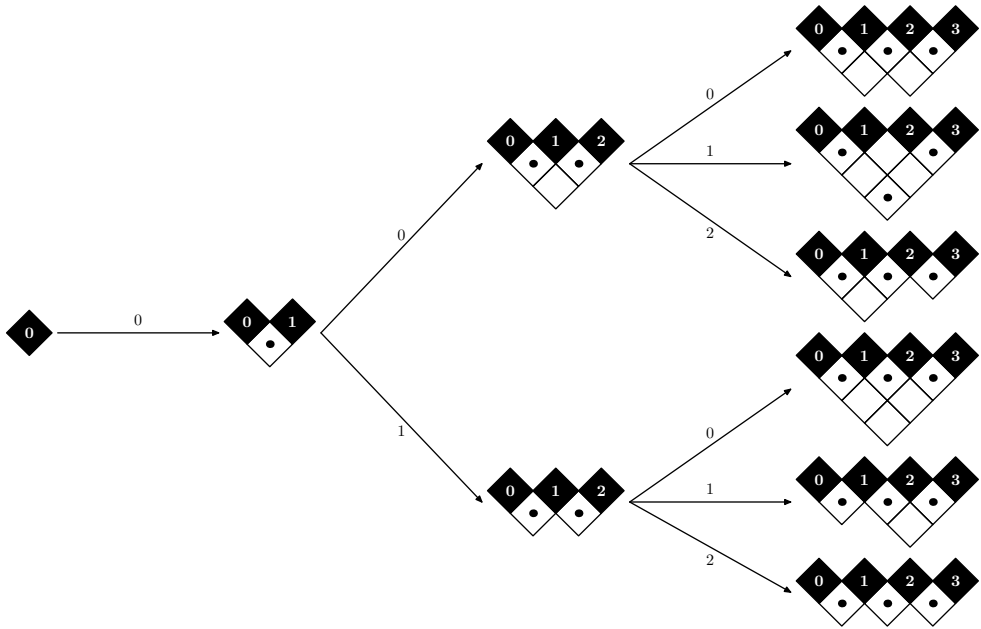}
\end{array}
$$
\caption{Generation tree for  Dyck tableaux of size at most $3$\label{fig:generation_tableau_dyck}}
\end{figure}

\begin{propriete}
The number of weighted Dyck tableaux of size $n$ is  $n!$.
\end{propriete}

\begin{proof}
It follows easily from the fact that the insertion procedure induces a bijection between couples
$(T,k)$, where $T$ is a Dyck tableau of size $n$ and $k$ is an integer with $0\le k\le n$,
and Dyck tableau $T'$ of size $n+1$.
\end{proof}

\section{Bijection between Dyck tableaux and permutations}
\label{sec:bij-permutations}

The insertion procedure can be used to construct a template of bijections between permutations and Dyck tableaux.
To obtain this template, to each tableau $T$ of size $n$, we label all the boxes of the basement from left to right by integers from $0$ to $n$.
Then, during the insertion procedure, we record in a table the labels of the chosen boxes.
We call this table the \gras{history table} of $T$.
For example, Table \ref{table:histoy_table} is the history table of the Dyck tableau of 
Figure \ref{fig:tableau_de_dyck_value_table}.
\begin{table}[H]
\begin{center}
\begin{tabular}{|c|c|c|c|c|c|}
\hline
insertion step & 1 & 2 & 3 & 4 & 5 \\
\hline
chosen box & 0 & 0 & 2 & 1 & 3 \\
\hline
\end{tabular}
\end{center}
\caption{A history table\label{table:histoy_table}}
\end{table}
\begin{figure}[H]
$$
\begin{array}{c}
\includegraphics[scale=.8]{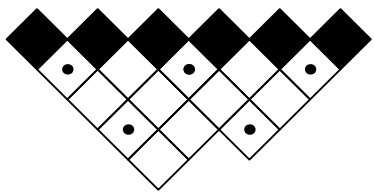}
\end{array}
$$
\caption{Dyck tableaux with history table of Table \ref{table:histoy_table} \label{fig:tableau_de_dyck_value_table}}
\end{figure}
We may see this history table as a path connecting the empty Dyck tableau 
to a Dyck tableau in the generation tree presented in Section \ref{sec:insertion}.
A vector $H\in\N^n$ is a history table if and only if 
$$\forall 1 \le j \le n,\ \ 0 \le H[j] < j.$$
Thanks to Propositions \ref{prop:inverse_de_algorithme_d'insertion} and  \ref{prop:generation_mot_de_dyck_value}, we know that two different Dyck tableaux give two different history tables.
We obtain a template of bijection, since we may choose different  interpretation of the history table as a permutation.

In this paper, we shall use the following interpretation.
Let $\sigma$ be a permutation of size $n$. The \gras{non-inversion table} of $\sigma$ is the table $NI_\sigma \in\N^n$ defined by
$$NI_\sigma[i]=\#\{j'<j:\ \sigma(j')<\sigma(j) \text{ and } \sigma(j)=i \}.$$
As an example, Table \ref{table:histoy_table} is the non-inversion table of the permutation $24153$.
It should be clear that the set of history tables of size $n$ coincides
with the set of non-inversion tables $NI_\sigma$ where $\sigma$
runs through permutations of size $n$.

By interpreting the history table of a Dyck tableau as the non-inversion 
table of a permutation,
we obtain a direct bijection from Dyck tableaux to permutations.
Throughout all this article, we will denote by $\TP$ this bijection.

We will now describe $\TP$ in a more simple and direct way.
Let $\sigma$ be a permutation of size $n$.
We first construct a basement of size $n+1$.
Then, we label all the columns of the basement from left to right with the entries of $\sigma$.
The bijection is obtained by iterating $n$ dot-insertions from the basement.
The $j$-th {\em dot-insertion} consists in the following two steps:
\begin{enumerate}[1)]
\item we add one dotted box in the column with entry $j$;
\item if the new dotted box is to the left of the dotted box added 
during step $j-1$, we add a ribbon between these two boxes. 
\end{enumerate}
Figure \ref{fig:algo_insertion_permutation} shows bijection $\TP$.
\begin{figure}[H]
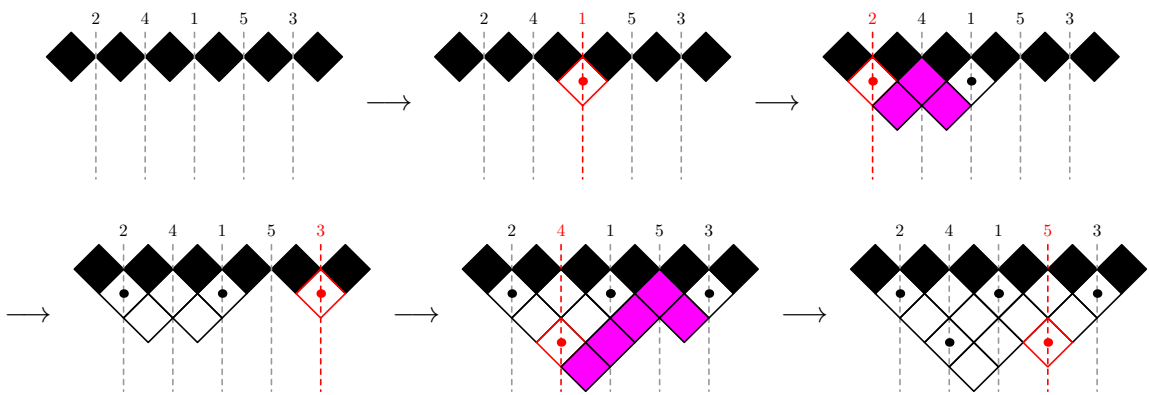

$$
\begin{array}{c}
\includegraphics[scale=.65]{images/bijection_permutations_tableaux_de_dyck_values.1}
\end{array}
\longrightarrow
\begin{array}{c}
\includegraphics[scale=.65]{images/bijection_permutations_tableaux_de_dyck_values.2}
\end{array}
\longrightarrow
\begin{array}{c}
\includegraphics[scale=.65]{images/bijection_permutations_tableaux_de_dyck_values.3}
\end{array}
$$
$$
\longrightarrow
\begin{array}{c}
\includegraphics[scale=.65]{images/bijection_permutations_tableaux_de_dyck_values.4}
\end{array}
\longrightarrow
\begin{array}{c}
\includegraphics[scale=.65]{images/bijection_permutations_tableaux_de_dyck_values.5}
\end{array}
\longrightarrow
\begin{array}{c}
\includegraphics[scale=.65]{images/bijection_permutations_tableaux_de_dyck_values.6}
\end{array}
$$
\caption{Bijection $\TP$ between permutations and Dyck tableaux\label{fig:algo_insertion_permutation}}
\end{figure}

It is important to note that applying bijection $\TP$ to a Dyck tableau 
$T$ gives the same result than interpreting the history table of $T$ as 
a non-inversion table.
To prove this fact, we just have to remark that:
\begin{itemize}
\myitem the insertion procedure can be followed from $\TP$ by removing at each dot-insertion the empty columns;
\myitem if $a(j)$ denotes the number of non empty columns to the left of the column
with entry $j$ after $j$ dot-insertion in bijection $\TP$, then one has:
$$a(j)=NI_{\TP(T)}[j].$$
\end{itemize}

It is clear that bijection $\TP$ is simpler than the insertion procedure, and that it gives a direct and elegant
solution to the problem studied by de Medicis and Viennot in \cite{dMV}, 
but a drawback of $\TP$ is that 
we have to wait the end of the procedure to obtain a Dyck tableau.
It is convenient with the insertion procedure that the current tableau 
during the construction is always a Dyck tableau.

\section{Generalized patterns}
\label{sec:patterns}

The aim of this section is to study some generalized patterns 
in permutations through the use of Dyck tableaux.
To this end, we choose to label the dots in a Dyck tableaux with 
respect to the insertion algorithm: a dotted box inserted at time $j$
gets the label $j$.
We shall consider patterns with constraints on their values, and 
not on their places, which is more usual in the literature. 
A \gras{pattern $2^+2$} of a permutation $\sigma$ is a sub-word 
$ab$ of $\sigma$ such that $a=b+1$.
A \gras{pattern $2^+12$} of a  permutation $\sigma$ is a sub-word 
$abc$ of $\sigma$ such that $b<c$ and $a=c+1$.
A \gras{pattern $1^+21$} of a permutation $\sigma$ is a sub-word 
$abc$ of $\sigma$ such that $a=c+1$ and $b>a$.
For example, in the permutation $2746153$, there are
\begin{itemize}
\myitem $4$ patterns $2^+2$ : $21$, $43$, $65$ and $76$; 
\myitem $3$ patterns $2^+12$: $746$, $615$ and $413$; 
\myitem $5$ patterns $1^+21$: $271$, $241$, $261$, $463$ and $453$.
\end{itemize}

It is more common in the literature (for example in \cite{CN,MJV}) 
to study descents and generalized patterns $31-2$ and $2-31$ 
instead of the patterns $2^+2$, $2^+12$ and $1^+21$ used in the present work.
This presentation is more convenient in our context, and it should be clear
that the map $\sigma\mapsto\sigma^{-1}$ sends our notions of generalized patterns
to the ones in previous works.

\begin{propriete}
\label{prop:bijection_motif_2p2}
The ribbons of a Dyck tableau $T$ are in bijection with the patterns 
$2^+2$ of $\TP(T)$.
Moreover, we get the pattern $2^+2$ itself by reading from left to right
the two labels that are linked by the ribbon.
\end{propriete}
As an example, the Dyck tableau on Figure \ref{fig:2p2_ruban} contains 
$4$ ribbons which correspond to the occurences $21$, $43$, $65$ and $76$ 
of the pattern $2^+2$ in the permutation $\sigma=2746153$.
\begin{figure}[H]
$$
\begin{array}{c}
\includegraphics{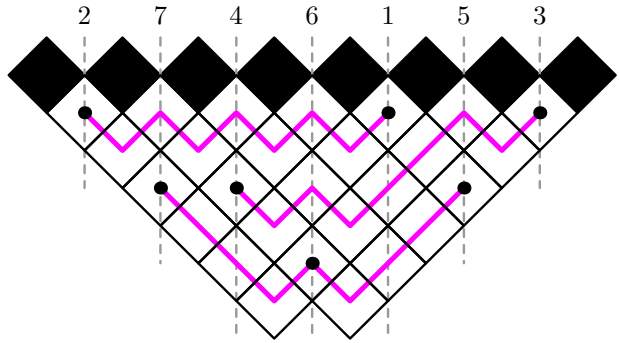}
\end{array}
$$
\caption{The ribbons represent the patterns $2^+2$ in the permutation
\label{fig:2p2_ruban}}
\end{figure}

\begin{proof}
Let $ab$ be a pattern $2^+2$ in the permutation $\sigma=\TP(T)$.
This means that the entry $a$ is inserted in $T$ immediately after $b$,
and on its left, whence there is a ribbon between $a$ and $b$ in $T$.

Conversely, if a ribbon in $T$ links the two dots labeled $a$ and $b$,
the inverse insertion algorithm removes $a$ immediately before $b$. 
Thus $a = b+1$.
\end{proof}

In a Dyck tableau $T$, the boxes placed below a dot (thus in a column)
are called \gras{shadow boxes}.
Figure \ref{fig:case_2p12} illustrates this notion.
\begin{figure}[H]
$$
\begin{array}{c}
\includegraphics{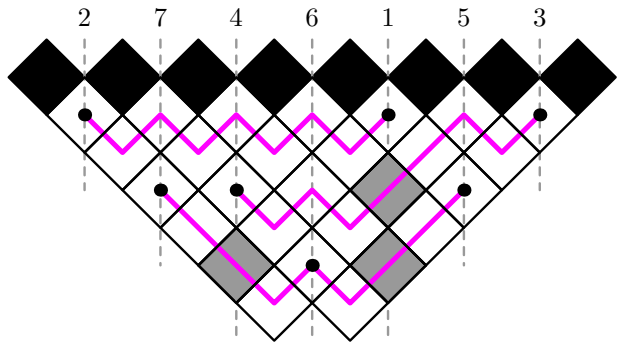}
\end{array}
$$
\caption{Shadow boxes of a tableau\label{fig:case_2p12}}
\end{figure}

Similarly, the boxes placed above a dot (thus in a column)
are called \gras{clear boxes}.
Figure \ref{fig:case_2p32} illustrates this notion.
\begin{figure}[H]
$$
\begin{array}{c}
\includegraphics{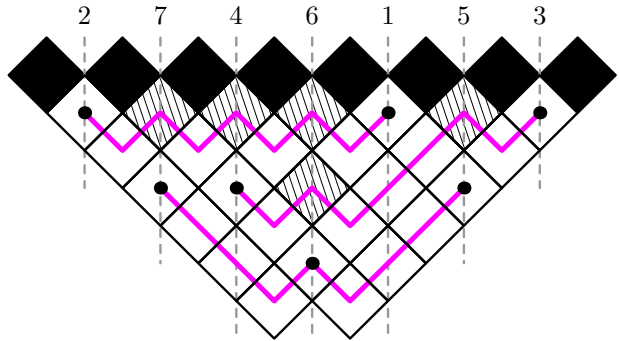}
\end{array}
$$
\caption{Clear boxes of a tableau\label{fig:case_2p32}}
\end{figure}

Every column contains three types of boxes: exactly one dot and 
possibly clear and shadow boxes ({\it cf.} Figure \ref{fig:case_2p12_2p32}).
\begin{figure}[H]
$$
\begin{array}{c}
\includegraphics{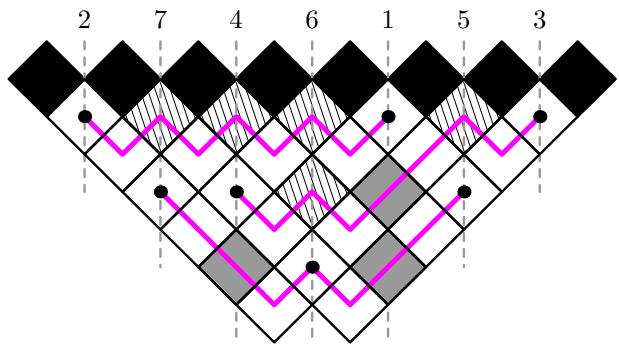}
\end{array}
$$
\caption{Clear and shadow boxes\label{fig:case_2p12_2p32}}
\end{figure}

\begin{propriete}
\label{prop:bijection_motif_2p12_2p32}
Let $T$ be a Dyck tableau and $\sigma=\TP(T)$.
Shadow boxes of $T$ are in bijection with patterns $2^+12$ 
of $\sigma$.
Clear boxes of $T$ are in bijection with patterns $1^+21$ 
of $\sigma$.
\end{propriete}

The Dyck tableau $T$ on Figure \ref{fig:case_2p12_2p32} contains
$5$ clear boxes corresponding to $271$, $241$, $261$, 
$463$, $453$ in the permutation $\sigma=2746153$,
and $3$ shadow boxes corresponding to $746$, $413$, $615$ 
in $\sigma$.

\begin{proof}
Let $abc$ be a pattern $2^+12$ in the permutation $\sigma$.
Thanks to bijection $\TP$, we know that the column with entry $b$ is between the columns of $a$ and $c$.
Proposition \ref{prop:bijection_motif_2p2} implies that the pattern
$2^+2$ corresponds to a ribbon which has been inserted after $b$.
Thus the column of $b$ intersects the ribbon and gives a shadow box.
It is obvious that two different occurences of the pattern $2^+12$
give two distinct shadow boxes: the triples of columns of $a$,
$b$ and $c$ are different.

Conversely, let us consider a shadow box. 
It is the intersection of a column labeled $b$ and a ribbon
of left and right endpoints labeled $a$ and $c$.
It should be clear that $abc$ is an occurence of the pattern $2^+12$.

The proof is the same for the pattern $1^+21$.
\end{proof}
A direct consequence of this proposition is the following corollary.
\begin{corollaire}
Let $\sigma$ be a permutation and $i$ a letter of $\sigma$.
The height of the dot labeled by $i$ in the Dyck tableau $T$
with $\TP(T)=\sigma$ is equal to the number of patterns $1^+21$
in $\sigma$ where $i$ is the entry $2$ in these patterns.
\end{corollaire}

\section{Bijection between Dyck tableaux and tree-like tableaux}
\label{sec:bij-TLT}

Tree-like tableaux are combinatorial ojects introduced in \cite{TLT}.
A {\em tree-like tableau} $B$ is a Ferrers diagram in which each box 
contains $0$ or $1$ dot (called respectively empty box or dotted box), with the following constraints:
\begin{enumerate}[1)]
\item the top-left-most box is dotted, and called the \gras{root} of $B$;
\item for every non-root dotted box $c$, there exists a dotted box either above $c$ in the same column, or to its left in the same row, {but not both}; 
\item every column and every row has at least one dotted box.
\end{enumerate}
The {\em size} of a tree-like tableau is defined as its number of dots.
Figure \ref{fig:tree-like_tableaux} shows an example of  tree-like tableau of size $9$.
\begin{figure}[H]
$$
\begin{array}{c}
\includegraphics[scale=.65]{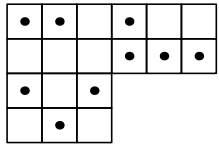}
\end{array}
$$
\caption{A tree-like tableau\label{fig:tree-like_tableaux}}
\end{figure}

An important property is that tree-like tableaux are endowed with a tree structure.
The tree structure can be obtained graphically by drawing two lines from every dot of $B$, one down and one to the right, and stopping them at the boundary.
We get in this way a binary tree with some edges crossings, which occur
necessarly inside boxes of the Ferrers diagram.
These boxes are called \gras{crossing boxes} of the tree-like tableau.
Figure \ref{fig:tree_structure_of_tree-like_tableaux} shows how to extract the tree structure of a tree-like tableau.
\begin{figure}[H]
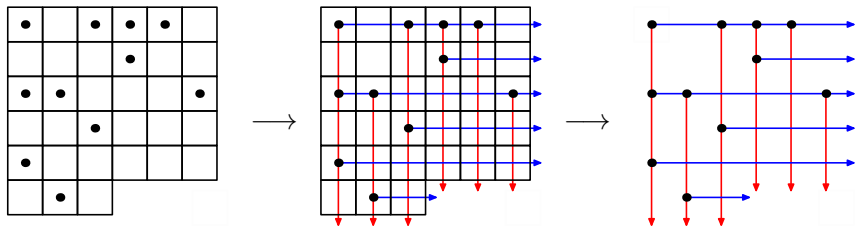

$$
\begin{array}{c}
\includegraphics[scale=.65]{images/tree_structure_of_tree-like_tableaux.1}
\end{array}
\longrightarrow
\begin{array}{c}
\includegraphics[scale=.65]{images/tree_structure_of_tree-like_tableaux.2}
\end{array}
\longrightarrow
\begin{array}{c}
\includegraphics[scale=.65]{images/tree_structure_of_tree-like_tableaux.3}
\end{array}
$$
\caption{The tree structure of a tree-like tableau\label{fig:tree_structure_of_tree-like_tableaux}}
\end{figure}

Like Dyck tableaux, tree-like tableaux are endowed with an insertion procedure \cite{TLT}.
We shall not give here details of this procedure, we just recall that there are $k+1$ positions (the edges placed at the bottom and right boundary) to insert a dotted box in a tree-like tableau of size $k$.
This explains why tree-like tableaux of size $n$ are, as Dyck tableaux of the same size, enumerated by $n!$.
We refer to \cite{TLT} for a detailed presentation of these objects.

Although Dyck tableaux and tree-like tableaux are two different notions which have been designed for different purposes,
their respective recursive structures are based on similar insertion procedures: in both cases, the key ingredients are
the notion of special box together with the ribbon addition which ensures to inversibility of the procedure.
Moreover, we will now explain that we have a canonical bijection between tree-like tableaux and Dyck tableaux.
To do this, we label the edges of any tree-like tableau of size $n$ following the boundary from left to right with integers from $0$ to $n$.
Then we may code a tree-like tableau $B$ with an {\em insertion table}
by recording at each insertion $k$ the label of the edge chosen to insert
the $k$-th dot.
We may see this table as the history table of a Dyck tableau $T$,
that is to say the non-inversion table of $\sigma=\TP(T)$. We get in this
way a bijection between tree-like tableaux $T$ and Dyck tableaux $D$ such that:
\begin{enumerate}[1)]
\item the number of dots in $T$ is the number of dots in $D$;
\item there is a ribbon between dot $j$ and dot $j+1$ in $T$ if and only if there is a ribbon between dot $j$ and dot $j+1$ in $B$ or if dot $j+1$ is the left son of dot $j$ in the tree extracted from the tree-like tableau;
\item crossing boxes of $B$ are in bijection with shadow boxes of $T$.
\end{enumerate}
Figure \ref{fig:bijection_tableaux_boises_tableaux_dyck_values} illustrates
this bijection.
\begin{figure}[H]
$$
\begin{array}{c}
\includegraphics{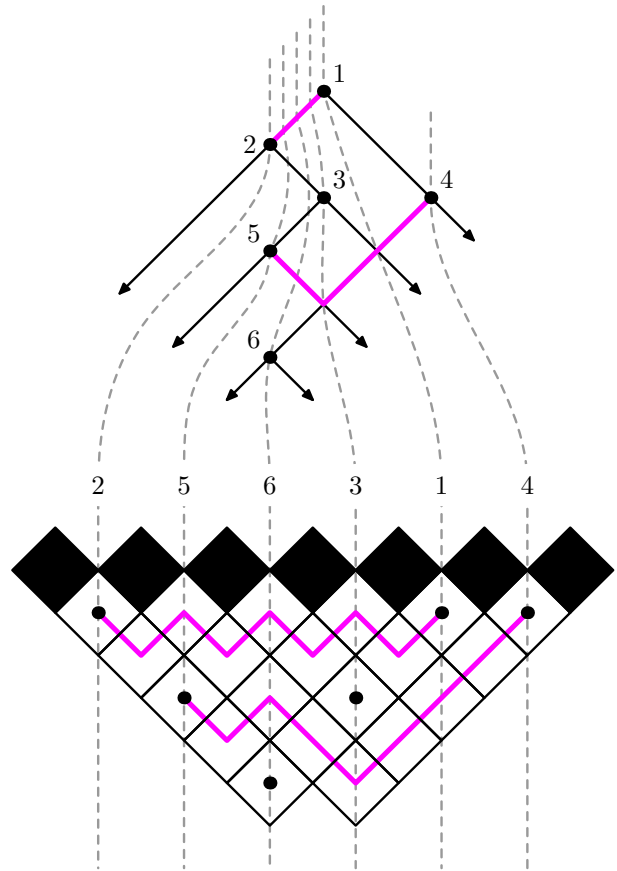}
\end{array}
$$
\caption{Bijection between tree-like tableaux and Dyck tableaux\label{fig:bijection_tableaux_boises_tableaux_dyck_values}}
\end{figure}

\begin{proof}
\begin{enumerate}
\item Evident from the insertion procedure.
\item Let $H$ be the insertion table of a tree-like tableau $B$, thus the history table of the Dyck tableau $T$.
When inserting $j$, we add a ribon between $j$ and $j+1$ in $T$ if and only if $H(j)\ge H(j+1)$.
We add a ribon between $j$ and $j+1$ in a tree-like tableau if and only if $H(j)>H(j+1)$.
Dot $j+1$ is the left son of dot $j$ if and only if $H(j)=H(j+1)$.
\item When inserting $j$ in $B$, the number of crossing boxes added is the length of the ribbon: $max(0,H(j)-H(j+1))$.
It is precisely the number of shadow boxes added for the insertion of $j$ in $T$.
\end{enumerate}
\end{proof}

\section{The shape of a Dyck tableau}
\label{sec:shape}

Thanks to the insertion algorithm, we are able to know the evolution 
of the path at the lower border of the Dyck tableau. In this section,
we are interested in reading this {\em shape} without using the insertion
algorithm, that is to say in reading it directly on the permutation.

We first define the right and left border. Consider the two steps in 
the Dyck path that are at the bottom of a column $i$ 
in the Dyck tableau. The left one is called 
\gras{left border}, the right one \gras{right border} at position $i$.
We denote by $lb(i)$ (\textit{resp.} $rb(i)$) the left 
(\textit{resp.} right) border at position $i$.
Figure \ref{fig:bord_gauche_bord_droit_tableau_dyck_value}
illustrates this notion: on this tableau, $lb(3)=D$
and $rb(3)=U$.
\begin{figure}[H]
$$
\begin{array}{c}
\includegraphics{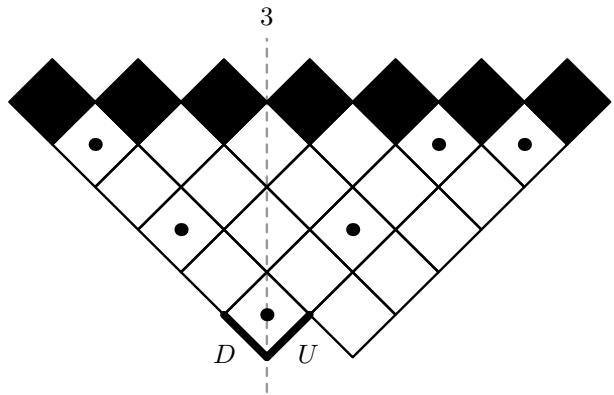}
\end{array}
$$
\caption{Left and right border in a Dyck tableau
 \label{fig:bord_gauche_bord_droit_tableau_dyck_value}}
\end{figure}

\begin{propriete}
\label{prop:invarience_du_bord}
Let $T$ be a Dyck tableau of size $n$ and $\sigma=\TP(T)$.
For a position $1\le i\le n$, $lb(i)$ and $rb(i)$ are known
after that $\sigma(i)+1$ (if exists) is inserted.
\end{propriete}

\begin{proof}
It suffices to observe that a ribbon added between two entries
$j=\sigma(i)$ and $j+1=\sigma(i')$ modifies the left
and right borders only at positions $i$ and $i'$.
\end{proof}

\begin{propriete}
\label{prop:bord_gauche_droit_tableau_dyck_value}
Let $T$ be a Dyck tableau of size $n$, $\sigma=\TP(T)$
and a position $1\le i\le n$. We let $j=\sigma(i)$.

The left border at position $i$ in $T$ is obtained by:
$$
lb(i) =
\left\{
\begin{array}{ll}
D & \text{if } j=n; \\
D & \text{if } j+1 \text{ is to the right of } j \text{ in } \sigma;\\
U & \text{if } j+1 \text{ is to the left of } j \text{ in } \sigma.
\end{array}
\right.
$$
The right border at position $i$ in $T$ is obtained by: 
$$
rb( i ) =
\left\{
\begin{array}{ll}
U & \text{if } j=1;\\
U & \text{if } j \text{ is to the right of } j-1 \text{ in } \sigma;\\
D & \text{if } j \text{ is to the left of } j-1 \text{ in } \sigma.
\end{array}
\right.
$$
\end{propriete}

For example, in $\sigma = 24153$, $4=\sigma(2)$ is to the left of $3$ and $5$. 
The left and right borders at position $2$ are both equal to $D$ (see
Figure \ref{fig:algo_insertion_permutation}).

\begin{proof}
Proposition \ref{prop:invarience_du_bord} allows to look only at the entries
 $j-1$, $j$ and $j+1$.
\begin{itemize}
	\myitem If $j=1$, there is no ribbon starting at position $i$ thus $rb( i ) =U$.

\myitem If $j \in [2,n]$, since the right step at position $i$ depends on whether there is a ribbon beginning at $j$, we consider the following: 

\begin{itemize}
	\item either $j$ is to the left of $j-1$, and $rb( i ) =D$, 
	\item or $j$ is to the right of $j-1$, and $rb( i ) =U$. 
\end{itemize}

	\myitem If $j=n$, there is no ribbon ending at position $i$ thus
$lb( i ) = D$

\myitem If $j \in [1,n-1]$, since the left step at position $i$ depends on whether there is a ribbon ending at $j$, we consider the following: 

\begin{itemize}
	\item either $j+1$ is to the right of $j$, and $lb( i )= D$,
	\item or $j+1$ is to the left of $j$, and $lb( i ) =U$.
\end{itemize}
\end{itemize}
\end{proof}

\section{The (RL/LR)-(minima/maxima) of a Dyck tableau}
\label{sec:min-max}

In this section, we want to interpret (RL/LR)-(minima/maxima)
of a permutation directly on its associated Dyck tableau.
Let $\sigma$ be a permutation of size $n$. 
We shall say that $j=\sigma(i)$ is a {\em right-to-left minimum}
(RL-minimum for short) if and only if 
$$\forall i'>i,\ \ \sigma(i)<\sigma(i').$$
We define in the same manner RL-maxima, LR-minima and LR-maxima.
By a slight abuse, we shall talk about (RL/LR)-(minima/maxima)
of a Dyck tableau.

We let $T$ denote a Dyck tableau of size $n$ and $\sigma=\TP(T)$.

\begin{propriete}
\label{prop:rl_minima}
There is a bijection between RL-minima of $\sigma$ and dotted boxes
 in $T$ at height $0$ with a right border equal to $U$.
\end{propriete}
Figure \ref{fig:rl_minima_tableaux_dyck_value} shows this proposition
for $\sigma = 3 1 4 2 8 5 7 6$ whose RL-minima 
are $6$, $5$, $2$ et $1$. 
\begin{figure}[H]
$$
\begin{array}{c}
\includegraphics{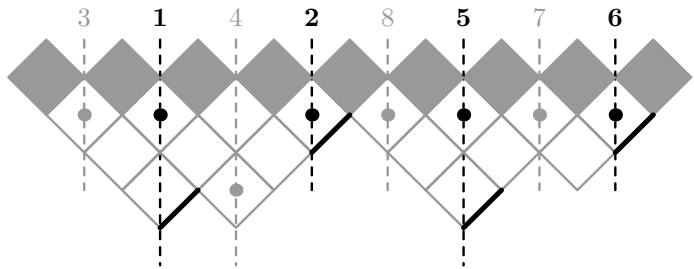}
\end{array}
$$
\caption{RL-minima of a Dyck tableau\label{fig:rl_minima_tableaux_dyck_value}}
\end{figure}
\begin{proof}
First observe that a RL-minimum $j$ can not have a ribbon above it 
(this would imply a smaller entry to its right), thus its height is $0$.
And since $j-1$ has to be to its left, $rb(j)=U$.

Conversely, let $j$ be an entry corresponding to a dotted box
at height $0$ with $rb(j)=U$. This implies that $j-1$ is to
the left of $j$. Since there is no ribbon above $j$, all entries smaller 
than $j$ are to the left of $j$.
\end{proof}

\begin{propriete}
\label{prop:lr_maxima}
There is a bijection between LR-maxima of $\sigma$ and dotted boxes
 in $T$ with maximal height and with a left border equal to $D$.
\end{propriete}
\begin{proof}
Exactly the same argument as Proposition \ref{prop:rl_minima}.
\end{proof}

For the next two results, we need to identify the entries $1$ and $n$
on the Dyck tableau.
\begin{propriete}
\label{prop:max}
The entry $n$ corresponds to the rightmost dotted box with maximal 
height and a left border equal to $D$. 
The entry $1$ corresponds to the leftmost dotted box with height $0$
and a right border equal to $U$. 
\end{propriete}
\begin{proof}
Since there is no ribbon under the dotted box corresponding  to entry $n$, 
its height is maximal. There is also no ribbon on its left, thus 
its left border is $D$.
Now, if two entries $j$ and $j'$
are such that $j>j'$ and $j$ to the left of $j'$,
we have either $lb(j')=U$ (if $j'+1$ is to the left of $j'$)
or at least one ribbon under $j'$ (if $j'+1$ is to the right of $j'$).
This implies that among points with maximal height
and with a left border equal to $D$,  the greatest is the rightmost.

The argument is the same for the entry $1$.
\end{proof}

It should be clear that for any $\sigma$, $n$ is a RL-maximum and
$1$ a LR-minimum.
 
\begin{propriete}
\label{prop:rl_maxima}
There is a bijection between RL-maxima of $\sigma$ different from $n$ 
and dotted boxes in $T$ with maximal height, with a left border equal to $U$
and placed to the right of $n$.
\end{propriete}

\begin{proof}
If $j$ is a RL-maximum, then
\begin{itemize}
\myitem there is no ribbon under $j$, thus it is at maximal height; 
\myitem $j+1$ is to the left of $j$, thus $lb(j)=U$;
\myitem $j$ has to be to the right of $n$.
\end{itemize}
Conversely, if $j$ satisfies the three previous properties,
$j+1$ has to be to the left of $j$, there is no ribbon 
under $j$ and since $n$ is to its left, any entry 
greater than $j$ has to be to its left also.
\end{proof}

\begin{propriete}
\label{prop:lr_minima}
There is a bijection between LR-minima of $\sigma$ different from $1$ 
and dotted boxes in $T$ with height $0$, with a right border equal to $D$ 
and placed to the left of $1$.
\end{propriete}
\begin{proof}
Same as Proposition \ref{prop:rl_maxima}.
\end{proof}

\section{Open questions and forthcoming work}
\label{sec:open}

As explained in the Introduction, a great motivation for this work
is to try to find new objects and statistics that allow a simpler
interpretation of the stationary distribution of the PASEP model
in full generality. 
This general model depends on five parameters 
$\alpha,\beta,\gamma,\delta,q$.
Corteel and Williams \cite{CW2} have given an 
explicit description of this distribution in terms of statistics
defined on staircase tableaux. The set $ST_N$ of staircase tableaux 
of size $N$
is partitioned into subsets $ST_N^\tau$, for $\tau$ any of the $2^N$
states of the model. For each staircase tableau $\T$,
Corteel and Williams introduce a weight $\w(\T)$ in the five parameters
$\alpha,\beta,\gamma,\delta,q$, then define 
$$Z_N^\tau=\sum_{ST_N^\tau}\w(\T)\ \ {\rm and}\ \ Z_N=\sum_{ST_N}\w(\T)$$
and prove that the stationary probability of a state $\tau$
is given by the quotient $Z_N^\tau/Z_N$.
A disadvantage of this description is that the definition of the weight
$\w$ is really complicated, in particular its dependence to the parameter
$q$. As a consequence, it seems out of reach to use this interpretation
to try to define a Markov chain directly on the combinatorial objects
(staircase tableaux) which should project to the PASEP.
Such a Markov chain has been obtained by Duchi and Schaeffer \cite{DS}
for $\gamma=\delta=q=0$ and more generally by Corteel and Williams 
\cite{CW-M} for $\gamma=\delta=0$.

In this perspective, the results of the present paper may be seen as
a first step: we may use Dyck tableaux
to describe the stationary probability of the PASEP for the special
case where $\gamma=\delta=0$. In this case, 
a formula for $Z_N^\tau(\alpha,\beta,0,0,q)$ was obtained in \cite{CW1},
then interpreted in terms of statistics on permutations in \cite{CN,MJV}.
These results may be translated as:
$$Z_N^\tau(\alpha,\beta,0,0,q) = \sum \w'(T)$$
where the sum is over a subset of Dyck tableaux
(which corresponds to the state $\tau$, 
and for which we have an explicit description)
and 
$$\w'(T)=\alpha^{-l(T)}\,\beta^{-r(T)}\,q^{s(T)}$$
with $l(T)$ the number of LR-minima of $T$,
$r(T)$ the number of RL-minima of $T$,
and $s(T)$ the number of shadow boxes in $T$.
Moreover, we have a conjectural description in the case
$\delta=0$ and $\beta=1$ in terms of Dyck tableaux, but there is 
still work to be done.

\bigskip
\noindent{\bf Acknowledgment.} The authors are very grateful to Xavier Viennot
for many useful explanations and comments, and for his interest in this work.
They also thank the anonymous referees for their valuable remarks.

\newpage

\end{document}